\documentclass[11pt,leqno,a4paper]{amsart}
\usepackage{amssymb,amsmath,amscd,graphicx,fontenc,amsthm,mathrsfs}
\usepackage[pdftex,bookmarks,colorlinks,breaklinks]{hyperref} 
\hypersetup{linkcolor=blue,citecolor=red}
\usepackage{multicol}
\usepackage[numbers]{natbib}

\newtheorem{theorem}{Theorem}[section]

\newtheorem{lemma}[theorem]{Lemma}
\newtheorem{proposition}[theorem]{Proposition}
\newtheorem{corollary}[theorem]{Corollary}

\theoremstyle{definition}
\newtheorem{remark}[theorem]{Remark}
\newtheorem{example}[theorem]{Example}

\newcommand\Tstrut{\rule{0pt}{2ex}}         
\newcommand\Bstrut{\rule[-1.5ex]{0pt}{0pt}}   
\newcommand{\pdm}{m_{\mathrm{proj}}}
\newcommand{\PGL}{\mathrm{PGL}}
\newcommand{\FF}{\mathbb{F}}

\newcommand{\RR}{\mathbb{R}}
\newcommand{\ZZ}{\mathbb{Z}}
\newcommand{\CC}{\mathbb{C}}

\newcommand{\G}{\mathscr{G}}
\newcommand{\GG}{\mathbf{G}}
\renewcommand{\SS}{\mathbf{S}}

\newcommand{\1}{\mathbf{1}}

\newcommand{\hgraph}[1]{{\mathsf{HG} }(#1)}

\newcommand{\rG}{\mathrm{G}}

\renewcommand{\AA}{\mathbf{A}}

\newcommand{\sing}[2]{\mathscr{G}_{#1}(#2)}
\newcommand{\esing}[2]{\overline{\mathscr{G}}_{#1}(#2)}

\DeclareMathOperator{\Id}{{\bf 1}}

\DeclareMathOperator{\Cay}{Cay}
\DeclareMathOperator{\SL}{SL}
\DeclareMathOperator{\GL}{GL}

\DeclareMathOperator{\E}{E}
\DeclareMathOperator{\Rep}{Rep}

\DeclareMathOperator{\rk}{rank}
\DeclareMathOperator{\codim}{codim}
\DeclareMathOperator{\Kl}{Kl}

\usepackage[top=2cm,bottom=2cm,right=2cm,left=2cm]{geometry}
\title{On the chromatic number of structured Cayley graphs}
\author[M. Bardestani]{Mohammad Bardestani}
\address{ DPMMS, Centre for Mathematical Sciences, Wilberforce Road, Cambridge, CB3 0WB.}
\email{mohammad.bardestani@gmail.com}
\author[K. Mallahi-Karai]{Keivan Mallahi-Karai}
\address{Jacobs University Bremen, Campus Ring I, 28759 Bremen, Germany.}
\email{k.mallahikarai@jacobs-university.de }
\keywords{Chromatic number; Finite groups of Lie type; Lang-Weil bound; Kloosterman sum.}
 \subjclass[2010]{20D60, 20G40}

\begin{document}
\maketitle
\begin{abstract}
In this paper, we will study the chromatic number of a family of Cayley graphs that arise from algebraic constructions. Using Lang-Weil bound and representation theory of finite simple groups of Lie type, we will establish lower bounds on the chromatic number of a large family of these graphs. As a corollary we obtain a lower bound for the chromatic number of certain Cayley graphs associated to the ring of $n\times n$ matrices over finite fields, establishing a result for the case of $\SL_n$ parallel to a theorem of Tomon~\cite{Tomon} for $\GL_n$. Moreover, using Weil's bound for Kloosterman sums we will also prove an analogous result for $\SL_2$ over certain finite rings.   
\end{abstract}
\section{Introduction}
Let $\rG$ be a group, and let $S$ be a symmetric subset of $\rG$, that is,  a set satisfying $S^{-1}=S$. Moreover assume that $\1\not\in S$ where $\1$ is the identity element of $\rG$. The Cayley graph of $\rG$ with respect to $S$, denoted by
$\Cay(\rG,S)$, is the graph whose vertex set is identified with $\rG$, and vertices $g_1, g_2 \in \rG$ are declared adjacent if and only if $g_1^{-1}g_2 \in S$. Recall also that the chromatic number of a graph $\mathcal{G}$, denoted by $\chi(\mathcal{G})$, is the least cardinal $c$ such that the vertex set $V(\mathcal{G})$ can be partitioned into $c$ sets (called color classes) such that no color class contains an edge in $\mathcal{G}$. 

The study of chromatic number of Cayley graphs and their subgraphs was first initiated by Babai~\cite{Babai3}. The focus of Babai's paper was on finding Cayley graphs of a given group with a {\it small} chromatic number. For instance, it was shown in~\cite{Babai3} that every solvable group $\rG$ has a generating set $S$ such that the chromatic number of $\Cay(\rG,S)$ is at most $3$. It is easy to see (\cite{Babai3}, Proposition 4.6) that $\chi(\Cay(\rG,S))=2$ for some generating set $S$ if and only if $\rG$ has a subgroup of index $2$. 

Bounding the chromatic number of Cayley graphs from below is a more subtle problem. 
Alon~\cite{Alon13} considered random Cayley graphs of arbitrary finite groups and established strong asymptotically almost sure lower bounds for their chromatic number. In the random model considered in this paper, $S$ is a randomly chosen subset of $\rG$ of a given cardinality $k$. Alon then establishes various lower bounds for the chromatic number of $\Cay(\rG,S)$ that hold with probability converging to $1$ as $n \to \infty$. In order for the bounds to be non-trivial, one needs $k  \gg \log n$. In the opposite direction, Alon also proved that if $\rG$ is abelian, and $k \ll \log \log n$, then with probability tending to $1$ as $n \to \infty$, the inequality $\chi( \Cay(\rG, S)) \le 3$ also holds. 

In this paper we will address similar problems in the case that the pair $(\rG,S)$ arises from an algebraic construction, and can thus be viewed  as highly structured. 
More precisely, let $\GG\subseteq\GL_n$ be a Chevalley group. Such groups are naturally obtained from a simple complex Lie algebra~\cite{Steinberg}. $\GG$ can also be viewed as a group scheme of finite type defined over $\ZZ$, which implies that there is a finite set $\{ f_i \}_{ i \in I}$
of polynomials with integer coefficients in variables $x_{ij}$ such that for any unital ring $R$, the common solutions of $\{ f_i \}_{i \in I}$ form a group, which is denoted by $\GG(R)$. 

The reader interested in concrete examples may consider the special case $\GG=\SL_n$, which is defined by the equation $\det (x_{ij})-1=0$. Note that the set of zeros of this polynomial over any unital ring $R$ defines the group $\SL_n(R)$, consisting of unimodular $n$ by $n$ matrices with entries in $R$. Let also $\tilde{\SS}$ be an affine subscheme of $\GG$ of finite type over $\ZZ$, namely 
\begin{equation}\label{Affine-sub}
 \tilde{\SS}:= \{ (x_{ij}) \in \GG: P_1(x_{ij})= \cdots = P_r(x_{ij})=0\},
\end{equation}
where $P_1, \dots, P_r$ are polynomials with integer coefficients in $x_{ij}$. Since $\GG$ and $\tilde{\SS}$ are defined over $\ZZ$, for any prime power $q$ we can consider the $\FF_q$-points of $\GG$ and $\tilde{\SS}$ denoted by $\GG(\FF_q)$ and $\tilde{\SS}(\FF_q)$. Here $\FF_q$ denotes the finite field with $q$ elements. We always assume that $\Id \not\in \tilde{\SS}$ and denote by 
\begin{equation}\label{G_q}
\G_{\GG,\SS}(\FF_q):=\Cay(\GG(\FF_q),\SS(\FF_q)),\qquad\quad \SS(\FF_q):= \tilde{\SS}(\FF_q) \cup \tilde{\SS}(\FF_q)^{-1},
\end{equation}
the Cayley graph of the group $\GG(\FF_q)$ with respect to the symmetrized set $\SS(\FF_q)$.

In order to make the setup clearer, let us give a concrete example. Let $n \ge 2$ and $q=p^f$ is an odd prime power. The {\bf non-singular graph} associated with $n$ and $q$ is defined by
\begin{equation}\label{Sing-def}
\sing{n}{\FF_q}=\Cay(\SL_n(\FF_q),\SS(\FF_q)),
\end{equation}
 where $\SS\subseteq \SL_n$ is defined by a single polynomial $P(x_{ij})=\det(I+(x_{ij}))$. This example can be obtained by taking $\GG=\SL_n$ and
 $\SS$ the subscheme defined by the single equation $P(x_{ij})=0$. 
 One can readily see that vertices $x,y \in \sing{n}{\FF_q}$ are adjacent if and only if $\det(x+y)=0$. These graphs are closely related to a family of graphs introduced by Anderson and Badawi~\cite{Anderson} under the name {\it regular graph}\footnote{Since the term {\it regular graph} has already a full-time job in graph theory, in interest of clarity, in this paper we will use the terms {\it non-singular graph} and {\it extended non-singular graph} for the specific graphs defined above.}, in which, instead of taking $\SL_n(\FF_q)$ as the vertex set, the larger group $\GL_n(\FF_q)$ has been used; we will call this graph the {\bf extended non-singular} graph associated with $n$ and $q$, and denote it by $\esing{n}{\FF_q}$. 
Evidently $\sing{n}{\FF_q}$ is an induced subgraph of $\esing{n}{\FF_q}$. 

These graphs have been the source of several other investigations. For instance, Akbari, Jamaali and Fakhari~\cite{AJF} showed that the clique number of $\esing{n}{\FF_q}$
is bounded by a universal constant depending on $n$, and independent of $q$ (for odd values of $q$). The following question has been posed by the authors of~\cite{AJF} in the 22nd British Combinatorics Conference 2009~\cite{Cam}:

\bigskip

{\it Problem 525: Given a field $F$ with characteristic other than $2$, define a graph whose vertex set is the set of invertible $n \times n$ matrices over $F$, by putting matrices $A$ and $B$ adjacent if $A + B$ is singular. It is known \cite{AJF} that the clique number of this graph is finite. Problem: Is the chromatic number of this graph finite?
 }
 
\bigskip

Interestingly enough, the answer to this question is negative. In fact,  Tomon~\cite{Tomon} proved that the chromatic number of the finite graph $\esing{n}{\FF_q}$ is at least $(q/4)^{\lfloor  n/2 \rfloor}$, under the assumption that $q=p^f$ is odd, which immediately implies that $\esing{n}{\overline{\FF}_p}$ is infinite. Here $ \overline{\FF}_p$ denotes an algebraic closure of $\FF_p$, and can be naturally viewed as the union of all finite fields $\FF_q$ for $q=p^f$.  Tomon's strategy of proof is to use Hoffman's spectral bound for the chromatic number of 
$\esing{2}{\FF_q}$, where, thanks to the smaller size of the underlying group, the spectrum of the adjacency matrix is more tractable. Having done this, he proceeds to reduce the general case to the case $n=2$. Tomon's negative answer makes this family of graphs particularly interesting in that, thanks to the universal upper bound on the clique size, they exhibit no obvious {\it local} obstruction to the existence of proper colorings with a bounded number of colors (as $q \to \infty$), while the lower bound on chromatic number reveals a {\it global} obstruction. Let us mention in passing that a version of this question over $\RR$ has also been dealt with in \cite{BMK} using a measure-theoretic version of Hoffmann's bound.

In the present paper we intend to show that the same phenomenon is exhibited by the more general family of graphs $\G_{\GG,\SS}(\FF_q)$. More precisely, we will show that as $q \to \infty$, we have $\chi(\G_{\GG,\SS}(\FF_q)) \to \infty$. In fact, similar to \cite{Tomon} we will establish a polynomial lower bound  for  $\chi(\G_{\GG,\SS}(\FF_q))$. Further, we will also show that the large chromatic number is not due to the existence of large cliques; in fact, as we will see, the clique number of $\G_{\GG,\SS}(\FF_q)$ is uniformly bounded as $q \to \infty$ (see Proposition \ref{no-large-clique}). Our first main result is the following:
 
\begin{theorem}\label{main}
Let $\GG\subseteq\GL_n(\CC)$ be a simple and simply-connected Chevalley group of rank $r=\rk \GG$ and dimension $d=\dim \GG$. Let $\tilde{\SS}$ be the affine subscheme of $\GG$, defined by~\eqref{Affine-sub}, and assume that $\tilde{\SS}$ is geometrically irreducible over the generic fiber, and let $m=\codim\tilde{\SS}$ be the codimension of $\tilde{\SS}$ in $\GG$. Then there exists a constant $C=C(d,r,m)$ such that for all but finitely many primes $p\geq 3$
\begin{equation}\label{main-bound}
 \chi(\G_{\GG,\SS}(\FF_q))\geq Cq^{\frac{r-m}{2}},
 \end{equation}
 where the characteristic of $\FF_q$ is $p$ and $\G_{\GG,\SS}(\FF_q)$ is the graph defined by~\eqref{G_q}. 
\end{theorem}

To flesh out the conditions imposed on $\tilde{\SS}$, let us consider again the 
case $\GG=\SL_n$. Assume that $\tilde{\SS}$ is defined by an integer polynomial  equation $f(x_{ij})=0$. Then $\tilde{\SS}$ will satisfy the condition if the polynomial $f(x_{ij})$, viewed as a polynomial over $\CC[x_{ij}]$, is irreducible. For instance, the polynomial $f(x_{ij})=\det(I+(x_{ij}))$ can be shown to satisfy this property.  

Let us say a few words about the strategy of proof. In literature, there are various spectral bounds for the chromatic number of graphs. These bounds are quite effective when the character theory of the underlying group is easily understood (for instance, when it is abelian), but they become prohibitively hard for more complicated groups. 
One feature of  the proof of Theorem~\ref{main} is that it provides a decent lower bound for the chromatic number by simply exploiting the quasirandomness of the underlying groups. One can thus view this argument as a soft spectral bound. Following Gowers~\cite{Gowers}, a finite group $\rG$ is called $D$-quasirandom if every non-trivial complex representation of $\rG$ has dimension at least $D$. Frobenius first observed that $\SL_2(\FF_p)$ is $(p-1)/2$-quasirandom. This result has been extended to all finite groups of Lie type by Landazuri and Seitz~\cite{LS}; see also~\cite{BKMS} for a uniform proof for adjoint Chevalley groups. As will soon see, a mixing inequality, which is a consequence of quasi-randomness, plays a key role in our proof. Aside from this, we will also take advantage of Lang-Weil bound for estimating the size of certain sets over $\FF_q$.

\begin{remark}\label{Singular-graph}
It would be interesting to see how sharp the bound \eqref{main-bound} is. In particular, we do not know what happens when 
$ \rk \GG \le \codim \tilde{\SS}$. Here is an interesting test case. Fix $n \ge 2$ and $0\le \ell \le n$. Define the graph $\sing{n,\ell}{\FF_q}$ with the vertex set $\SL_n(\FF_q)$, in which two matrices $x,y$ form an edge if and only if
$\rk(x+y) \le n-\ell$. For $\ell=1$, we will recover $\sing{n}{\FF_q}$. Notice that all the connected components of the algebraic variety $\{x\in \SL_n(\CC): \rk(I+x) \le n-\ell\}$ have  codimension $\ell^2$ and so by  Theorem~\ref{main} we obtain
\begin{equation}\label{l=1}
\chi( \sing{n,\ell}{\FF_q}) \gg_{n,\ell} q^{ \frac{n-\ell^2-1}{2}}.
\end{equation}
This bound is non-trivial as long as $\ell \ll \sqrt{n}$, but we do not know what happens even for $\ell=n-1$.
\end{remark}

\begin{remark}
Groups with strong expansion properties tend to have large chromatic number. Prime examples of this kind are Ramanujan graphs, which were first constructed in \cite{LPS} by Lubotzky, Phillips, Sarnak, and, independently, by Margulis \cite{Margulis}. Note that these graphs have bounded degree. Exploiting the spectral gap and using Hoffman's bound, one can show that they also have large chromatic numbers. Although these graphs have the advantage of having a bounded degree independent of $q$, but at the same time, they are less flexible and more difficult to construct. For details we will refer the reader to \cite{Lubotzky}.
\end{remark}

Let us explain the reason for considering simple algebraic groups by showing an obvious obstruction to the growth of the chromatic numbers with $q$. 
\begin{example}[Abelian quotients as an obstruction] 
Let $\AA$ be either $\mathbb{G}_a$ or $\mathbb{G}_m$ and let $\SS$
 be a proper Zariski closed subset of $\AA$ defined over $\ZZ$, and assume that $\Id\not\in \SS$. For simplicity, assume that $\SS=\{ \pm a\}$ in the case of the additive group, and 
$\SS= \{ a^{\pm 1} \}$ in the case of the multiplicative group. It is easy to see that (except for a finite number of characteristics) for any $q$, each connected component of the graph  $\G_{\AA,\SS}(\FF_q)$ is isomorphic to a cycle and so $\chi(\G_{\AA,\SS}(\FF_q))\leq 3$.  More generally, let $\phi: \GG \to \AA$ be a nontrivial morphism of algebraic groups defined over $\ZZ$. Set $\SS_1=\phi^{-1}(\SS)$. We easily see that $\chi(\G_{\GG, \SS_1}(\FF_q))\leq \chi(\G_{\AA,\SS}(\FF_q))$  and hence its chromatic number is also bounded by $3$.  Note also that $\SS_1=\phi^{-1}(\SS)$ has codimension one in $\GG$. This example shows that Theorem~\ref{main} is not unconditionally true for all algebraic
groups. It would be interesting to know if other obstructions exist. 
\end{example}

As an immediate corollary of Theorem \ref{main} we have
\begin{corollary}\label{infinity}
Let $\GG$ and $\tilde{\SS}$ be as in Theorem \ref{main} and assume that $\rk\GG> \codim\tilde{\SS}$. Then for all but finitely many primes $p\geq 3$ 
\[ \chi(  \G_{\GG, \SS}( \overline{\FF}_p))= \infty, \]
where $\overline{\FF}_p$ is the algebraic closure of $\FF_p$ and $p\geq 3$.
\end{corollary}

\begin{remark}
When specialized to the case of non-singular graphs defined in~\eqref{Sing-def}, 
Corollary \ref{infinity} shows that the chromatic number of the $\sing{n}{ \overline{\FF}_q  }$
is infinity. This establishes a strengthening of a theorem of Tomon \cite{Tomon}, which establishes an analogous result for the extended singular graphs. At a more quantitative level, Theorem \ref{main}, when restricted to the non-singular graphs, establishes a lower bound of order $q^{ \frac{n-2}{2}}$ for the chromatic number of the singular graphs over $\FF_q$. As mentioned above, Tomon's result applies to the extended non-singular graphs (which are about $q$ times larger in size) and yields the bound $q^{ \left\lfloor \frac{n}{2} \right\rfloor}$. In this respect, the results cannot be directly compared. However, what makes Theorem~\ref{main} useful is its robustness, it that it can be applied to a large class of Cayley graphs. 

\end{remark}

It is noteworthy that the polynomial growth of the chromatic number is not due to the existence of a large clique.

\begin{proposition}\label{no-large-clique}
Suppose $\GG$ and $\tilde{\SS}$ are as in the statement of Theorem~\ref{main}. Then there exists a constant $C=C({\GG, \tilde{\SS}})$ such that for a sufficiently large prime $p$ we have $\omega(  \G_{\GG,\SS}(\FF_q))\leq C$, where  $\omega(\G)$ denotes the size of the largest clique in  graph $\G$ and $\FF_q$ is the finite field with $q=p^f$ elements. 
\end{proposition}

As pointed out above, our method does not give any non-trivial bound when $\rk \GG \le \codim \tilde{\SS}$. In some cases, one can invoke spectral bounds to give relatively sharp estimates. Our next theorem is a special case. Recall that the non-singular graph $\sing{2}{\FF_q}$ is a graph with the vertex set $\SL_2(\FF_q)$, in which two matrices $x,y$ form an edge if and only if $\det(x+y)=0$.
\begin{theorem}\label{main-SL_2} Let $\FF_q$ be a finite field of cardinality $q$, and characteristic $p\geq 3$. Then we have
\begin{equation}\label{chor-SL2}
q+1\leq \chi(\sing{2}{\FF_q}) \le 8(q+1). 
\end{equation}
\end{theorem} 
\begin{remark} As indicated above, Tomon~\cite[Theorems 2.1 and 2.5]{Tomon} has proven that for all prime powers $q$, the inequalities $q/4\leq \chi(\esing{2}{\FF_q} )\leq 4q(q+1)$ hold. The method of~\cite{Tomon}, however, does not seem to be applicable to the subgraph $\sing{2}{\FF_q}$, and will only yield a trivial bound. Moreover by modifying the proof of Theorem~\ref{main-SL_2} one can show that $\chi(\esing{2}{\FF_p})\ll p$ when $p\equiv 3\pmod{4}$. Therefore one might expect that the correct order of $\chi(\esing{2}{\FF_q})$ should be $q$ rather than $q^2$.
\end{remark}
\begin{remark} When $-1$ is a quadratic non-residue in the finite field $\FF_q$, the upper bound in~\eqref{chor-SL2} can be improved to $2(q+1)$. It is clear from Theorem~\ref{main-SL_2} that $\chi( \sing{2}{\FF_q})$ grows linearly with $q$. It would be 
interesting to obtain sharper lower and upper bounds. For instance, one can ask if 
$\lim_p \chi( \sing{2}{\FF_p})/p$ exists. Similarly,~\eqref{chor-SL2} may suggest that 
$\chi( \sing{n}{\FF_q})$ has the order $q^{n-1}$. 
\end{remark}
The non-singular graph $\sing{n}{\FF_q}$ can also be defined over finite rings. $\sing{n}{\ZZ/p^r\ZZ}$ is a graph with the vertex set $\SL_n(\ZZ/p^r\ZZ)$, in which two matrices $x,y$ form an edge if and only if $\det(x+y)=0$. We now discuss lower bound for $\chi(\sing{2}{\ZZ/p^r\ZZ})$. Although it is conceivable that the character theory of $\SL_2(\ZZ/p^r\ZZ)$ can be employed for this purpose, lengthy computations are bound to arise. We will circumvent the complications by finding certain simpler Cayley graphs inside $\sing{2}{\ZZ/p^r\ZZ}$. The spectral bounds for the chromatic number of these graphs turn out to be intimately related to certain Kloosterman sums, for which classical estimates exist. Using this strategy we can prove the following theorem.
\begin{theorem}\label{Klo-SL} Let $p\geq 5$ be a prime number. Then 
$$
\frac{\sqrt{p}}{4}\leq \chi(\sing{2}{\ZZ/p^r\ZZ}) \leq 8(p+1).
$$
\end{theorem}
This paper is organized as follows: in Section \ref{QR} we will give a proof of Theorem \ref{main}. Sections~\ref{Cay} and~\ref{sec-hyperbola} are devoted to the proof of Theorems~\ref{main-SL_2} and~\ref{Klo-SL}.
\section{Quasirandom groups}\label{QR}
Quasirandom groups were introduced by Gowers~\cite{Gowers} in order to answer a question of Babai and S{\'o}s on product-free sets in finite groups. We recall that a finite group $\rG$ is called $D$-quasirandom if every non-trivial complex representation of $\rG$ has dimension at least $D$. One of the main results in~\cite{Gowers} is the following mixing inequality: 
\begin{theorem}\label{Gow}
Let $\rG$ be a finite $D$-quasirandom group. If 
$A,B,C \subseteq \rG$ such that 
$$|A| |B| |C| > |\rG|^3/D,$$
then the set $AB \cap C$ is non-empty.
\end{theorem}
Gowers' proof, as well as the proof given later by Babai, Nikolov and Pyber~\cite{Pyber}, is based on spectral analysis of graphs. A Fourier analytic proof of this theorem can also be found in~\cite{Breuillard}.
\begin{lemma}\label{main-prop} Let $\rG$ be a finite $D$-quasirandom group and let $S$ be a symmetric subset of $\rG$ with the associated Cayley graph $\Cay(\rG,S)$. Assume that $\1\not\in S$, then 
\begin{equation}
\chi(\Cay(\rG,S))\geq \sqrt{\frac{D|S|}{|\rG|}}\; .
\end{equation} 
\end{lemma}
\begin{proof} 
let $\kappa$ be the chromatic number of $\Cay(\rG,S)$. Hence $\rG$ can be partitioned into $\kappa$ subsets $A_1, \dots, A_\kappa$ such that $xy^{-1}\not\in S$ when $x,y \in A_i$ for all $1 \le i \le \kappa$. There exists $1 \le i \le \kappa$ such that $A_i$ has size at least $|\rG|/\kappa$. Set $A= A_i$, $B=A_i^{-1}$, and $C=S$. From the definition of chromatic number, we have $AB\cap S=\emptyset$. Therefore, by Theorem~\ref{Gow}, we conclude
\[\frac{|\rG|^2|S|}{\kappa^2} \leq|A||B||S|=|A|^2 |S|\leq \frac{|\rG|^3}{D},\]
which completes the proof.
\end{proof}
Next we prove the following simple lemma which will be useful later. For a given group $\rG$, the minimal dimension of non-trivial irreducible representations of $\rG$ is denoted by $m(\rG)$ and $\pdm(\rG)$ denotes the minimal dimension of non-trivial irreducible projective representations of $\rG$.
\begin{lemma}\label{G/Z} Let $\rG$ be a perfect group. Then $m(\rG)\geq \pdm(\rG/\mathrm{Z}(\rG))$, where $\mathrm{Z}(\rG)$ is the center of $\rG$.  
\end{lemma}
\begin{proof} Let $\rho: \rG\to \GL_n(\CC)$ be a non-trivial irreducible representation. Since $\rG$ is perfect, and thus does not have any non-trivial one dimensional representation, we obtain a non-trivial projective representation $\bar{\rho}: \rG\to \PGL_n(\CC)$. Moreover from irreducibility of $\rho$ we conclude that $\mathrm{Z}(\rG)\subseteq \ker\bar{\rho}$ and so we obtain a non-trivial irreducible projective representation $\bar{\rho}:\rG/\mathrm{Z}(\rG)\to \PGL_n(\CC)$ and so $n\geq \pdm(\rG/(\mathrm{Z}(\rG)))$. 
\end{proof}
\begin{proof}[Proof of Theorem~\ref{main}] Let $d=\dim \GG$, $r=\rk \GG$, $m=\codim \tilde{\SS}$ and $q=p^f$. 
By applying Schwarz-Zippel bound~\cite[Proposition A.2]{Tao} we observe that $|\GG(\FF_q)|\leq c_1 q^d$ where $c_1$ is independent of $q$. Since $\tilde{\SS}$
is geometrically irreducible, then by Bertini-Noether~\cite[Corollary 10.4.3]{Jarden} we can conclude that, for all but finitely many primes $p$, the variety $\tilde{\SS}(\FF_q)$ is geometrically irreducible with the same dimension as $\tilde{\SS}$. Hence by Lang-Weil bound~\cite{LangWeil}, we obtain
$|\SS(\FF_q)|\geq c_2 q^{d-m}$,
where $c_2$ is independent of $q$ and $\SS(\FF_q)= \tilde{\SS}(\FF_q) \cup \tilde{\SS}(\FF_q)^{-1}$. Therefore
\[ \frac{|\SS(\FF_q)|}{|\GG(\FF_q)|}\geq \frac{c_3}{q^m}, \]
where $c_3$ is independent of $q$. Moreover, for all but finitely many primes $p$,  the finite group $\GG(\FF_q)$ is perfect and $\GG(\FF_q)/\mathrm{Z}(\GG(\FF_q))$ is simple~\cite[Theorem 24.17]{Malle}. By Landazuri and Seitz theorem~\cite{LS}, we have 
$\pdm(\GG(\FF_q)/\mathrm{Z}(\GG(\FF_q)))\geq c_4q^{\rk(\GG)},$
where $c_4$ is independent of $q$ and so by Lemma~\ref{G/Z} we can conclude that $\GG(\FF_q)$ is $O(q^{r})$-quasirandom. Since $\Id\not\in\tilde{\SS}$, then by applying Lemma~\ref{main-prop} we have 
$$\chi(\G_{\GG,\SS}(\FF_q))\gg q^{ \frac{r-m}{2}}.$$
\end{proof}

\begin{proof}[Proof of Proposition \ref{no-large-clique}]
The proof of this proposition 
is inspired by the proof of an analogous result in~\cite{AJF}. First, observe that if $\tilde{\SS}_1 \subseteq \tilde{\SS}_2$, then $\G_{\GG,\SS_1}(\FF_q)$ is a subgraph of $\G_{\GG,\SS_2}(\FF_q)$.
Hence, without loss of generality, we may assume that $\tilde{\SS}$ is defined by a single equation $P(x_{ij})=0$, where $P(x_{ij})$ is a polynomial with integer coefficients in $n^2$ variables $x_{ij}, 1 \le i, j \le n$. Since $\1\not\in\tilde{\SS}$ then we can assume that $P(\1) \neq 0$. Let $p>|P(\1)|$ be an arbitrary prime number. 
Then, in a field of characteristic $p$, we have $P(\Id) \neq 0$. Consider now a clique $\mathcal{C} $ in graph  $\G_{\GG,\SS}(\FF_q)$, and for each $A \in \mathcal{C}$, set $f_A(x_{ij})= P(A^{-1}X)\in \FF_q[x_{ij}]$, where $X=(x_{ij})$, and we have implicitly identified the matrix $A^{-1}X$ with its $n^2$ entries.  We claim that polynomials $\{ f_A(X)\}_{ A \in \mathcal{C}}$ are linearly independent over $\FF_q$. Assume
\begin{equation}\label{linear-dependence}
\sum_{i=1}^{| \mathcal{C}|}  \lambda_i f_{A_i}(x_{ij})=0, 
\end{equation}
where $ \lambda_i \in \FF_q$. Note that since $\mathcal{C}$ is a clique, we have $A_i^{-1}A_j \in \SS(\FF_q)$ for all $ i \neq j$. In view of this, after substituting $A_j$, $1 \le j \le | \mathcal{C}|$ for $X$ in \eqref{linear-dependence} we obtain 
\[ 0= \sum_{i=1}^{| \mathcal{C}|}   \lambda_i P(A_i^{-1}A_j)=  \lambda_j P(\Id). \]
Since $P(\Id) \neq 0$, we obtain $ \lambda_j=0$. Now if $P$ has degree $d$, then all polynomials $f_{A_i}$ belong to the vector space of polynomials of degree at most $d$ in $n^2$ variables $x_{ij}$ over $\FF_q$, whose dimension $D(n,d)$ over $\FF_q$ is fixed {\it independent} of $q$.  Linear independence of these polynomials implies that 
$| \mathcal{C}| \le D(n,d)$, which establishes the claim. 
\end{proof}

\section{Cayley graphs and their spectra}\label{Cay} We first recall some well-known facts from algebraic graph theory. The following spectral bound for the chromatic number of graphs is due to Hoffman and we refer the reader to~\cite[Theorem 7, page 265]{Bollobas} for its proof.
\begin{lemma}\label{Hoffman}  Let $\mathcal{G}$ be a non-empty graph with $n$ vertices.  Then 
\begin{equation}\label{Hoffman-eq}
\chi(\mathcal{G})\geq 1-\lambda_0/\lambda_{n-1}.
\end{equation}
where $\lambda_0\geq \lambda_1\geq\dots\geq \lambda_{n-1}$ is the spectrum of the adjacency matrix of $\mathcal{G}$.
\end{lemma}
In some cases it will be more convenient to use the following lower bound~\cite[1.5.4. Corollary]{Sarnak}.
\begin{lemma}\label{Sar} Let $\mathcal{G}$ be a finite, connected, $\ell$-regular graph on $n$ vertices, without loop. Then
$$
\chi(\mathcal{G})\geq \frac{\ell}{\max\left\{|\lambda_1|,|\lambda_{n-1}|\right\}}.
$$
\end{lemma}
Now let $\rG$ be a group and let $S\subseteq \rG$ be a symmetric set. Note that the Cayley graph $\Cay(\rG,S)$ is $|S|$-regular and is connected if and only if $S$ generates $\rG$.
The following theorem due to Babai~\cite{Babai} and Diaconis-Shahshahani~\cite{Shahshahani}, describes the spectrum of the adjacency matrix of $\Cay(\rG,S)$ using the character theory of $\rG$.
\begin{theorem}\label{Shah}
Let $\rG$ be a finite group and $S$ a symmetric subset which is stable
under conjugation. Let $A$ be the adjacency matrix of the graph $\Cay(\rG,S)$. Then the eigenvalues of $A$ are given by
$$
\lambda_\rho=\frac{1}{\dim(\rho)}\sum_{s\in S}\chi_\rho(s),
$$
as $\chi_\rho$, the character of the representation $\rho$, ranges over all irreducible characters of $\rG$. Moreover, the multiplicity of $\lambda_\rho$ is $\dim(\rho)^2$.
\end{theorem}
We now turn to the graph $\sing{2}{\FF_q}$, $q=p^f$ for an odd prime $p$. We remark that the method of the previous section is not applicable to $\sing{2}{\FF_q}$. Instead we use representation theory to obtain a lower bound for the chromatic number of this graph. Let $\E_{2,q}$ denote the set of matrices in $\SL_2(\FF_q)$ which have $-1$ as an eigenvalue. Since $\E_{2,q}$ is a union of conjugacy classes, we can use the Jordan canonical form to give a simple description of the set $\E_{2,q}$. Let $\nu$ be a generator of the cyclic group $\FF_q^*$. Any matrix in $\SL_2(\FF_q)$ with an eigenvalue $-1$ is either $-I$, where $I$ is the identity matrix, or is conjugate in $\SL_2(\FF_q)$ to one of the following matrices
$$
T_1:=\begin{pmatrix}
-1 & 0\\
-1 & -1
\end{pmatrix}, \qquad T_2:=\begin{pmatrix}
-1 & 0\\
-\nu & -1
\end{pmatrix}.
$$ 
Hence $\E_{2,q}=\{-I, (T_1), (T_2)\}$, where $(T_1)$ and $(T_2)$ denote the conjugacy classes of $T_1$ and $T_2$. It is easy to see that each of these conjugacy classes has $(q^2-1)/2$  elements and so $|\E_{2,q}|=q^2$. 
We recall that $\SL_2(\FF_q)$ is generated by unipotent matrices. Notice that the subgroup generated by $\E_{2,q}$ contains  
\begin{equation}\label{anti-uni}
\begin{pmatrix}
-1 & 0\\
a & -1
\end{pmatrix}^2= \begin{pmatrix}
1 & 0\\
2a & 1
\end{pmatrix}   ,\qquad \begin{pmatrix}
-1 & b\\
0 & -1
\end{pmatrix}^2= \begin{pmatrix}
1 & 2b\\
0 & 1
\end{pmatrix}, \qquad a,b\in\FF_q.
\end{equation}
Hence $\E_{2,q}$ generate $\SL_2(\FF_q)$ from which it follows that the graph $\sing{2}{\FF_q}=\Cay(\SL_2(\FF_q),\E_{2,q})$ is a $q^2$-regular connected graph. Let $A$ be the adjacency matrix of $\sing{2}{\FF_q}$ with eigenvalues 
$$
\lambda_0\geq \lambda_1\geq\dots\geq \lambda_{n-1},\qquad n=|\SL_2(\FF_q)|.
$$
Since $\sing{2}{\FF_q}$ is a $q^2$-regular connected graph then we have
\begin{equation}\label{lambda_{n-1}}
q^2= \lambda_0> \lambda_1\geq 0>\lambda_{n-1}.
\end{equation}
 We now invoke Lemma~\ref{Hoffman} to find a lower bound for the chromatic number of $\sing{2}{\FF_q}$. In order to apply this lemma we need to estimate the size of the eigenvalues of the adjacency matrix, which by Theorem~\ref{Shah} are given by
\begin{equation}\label{shah-calc}
\lambda_{\rho}=\frac{1}{\dim(\rho)}\sum_{s\in \E_{2,q}}\chi_\rho(s)=\frac{1}{\dim(\rho)}\left(\chi_\rho(-I)+\frac{q^2-1}{2}\chi_\rho(T_1)+\frac{q^2-1}{2}\chi_\rho(T_2)\right),
\end{equation}
where $\chi_\rho$, the character of the representation $\rho$, ranges over all irreducible characters of $\SL_2(\FF_q)$.
Representations of $\SL_2(\FF_q)$ have been studied by Frobenius and Schur. For more details we refer the reader to~\cite[Section 38]{Dorn}. 
From~\eqref{lambda_{n-1}} and~\eqref{shah-calc}, to evaluate $\lambda_{n-1}$ we only need to know the values of non-trivial characters at $-I, T_1$ and $T_2$.

 Denote $\varepsilon=(-1)^{(q-1)/2}$. Then for $1\leq i\leq (q-3)/2$ and $1\leq j\leq (q-1)/2$ we have the following table (see~\cite[Theorem 38.1]{Dorn}): 
\[
\bigskip
\begin{array}{|l|c|c|c|c|}\hline
\Rep & \dim & -I & T_1 &  T_2 \Tstrut\Bstrut\\ \hline
\psi    & q & q  & 0 & 0 \Tstrut\Bstrut\\
\chi_i    & q+1  & (-1)^i(q+1)  & (-1)^i  & (-1)^i \Tstrut\Bstrut\\
\theta_j    &q-1  &(-1)^j(q-1)  & (-1)^{j+1}  & (-1)^{j+1} \Tstrut\Bstrut\\
\xi_1    &\frac{1}{2}(q+1)  &\frac{1}{2}\varepsilon(q+1)  &\frac{1}{2}\varepsilon (1+\sqrt{\varepsilon q})  & \frac{1}{2}\varepsilon (1-\sqrt{\varepsilon q})\Tstrut\Bstrut \\ 
\xi_2    &\frac{1}{2}(q+1)  &\frac{1}{2}\varepsilon(q+1)  &\frac{1}{2}\varepsilon (1-\sqrt{\varepsilon q})  & \frac{1}{2}\varepsilon (1+\sqrt{\varepsilon q}) \Tstrut\Bstrut\\ 
\eta_1    &\frac{1}{2}(q-1)  &-\frac{1}{2}\varepsilon(q-1)  & \frac{1}{2}\varepsilon (1-\sqrt{\varepsilon q})  & \frac{1}{2}\varepsilon (1+\sqrt{\varepsilon q}) \Tstrut\Bstrut\\ 
\eta_2    &\frac{1}{2}(q-1)  &-\frac{1}{2}\varepsilon(q-1)  & \frac{1}{2}\varepsilon (1+\sqrt{\varepsilon q})  & \frac{1}{2}\varepsilon (1-\sqrt{\varepsilon q}) \Tstrut\Bstrut\\ \hline
\end{array}
\]
From~\eqref{shah-calc}, the above table and a simple calculation we obtain the following equalities:
\begin{equation}\label{simple-cal}
\lambda_\psi=1,\qquad \lambda_{\chi_i}=(-1)^{i}q,\qquad \lambda_{\theta_j}=(-1)^{j+1}q,\qquad \lambda_{\xi_1}=\lambda_{\xi_2}=\lambda_{\eta_1}=\lambda_{\eta_2}=\varepsilon q.
\end{equation}
With these preliminaries, we are now ready to prove Theorem~\ref{main-SL_2}.
\begin{proof}[Proof of Theorem~\ref{main-SL_2}] Using~\eqref{simple-cal} along with Theorem~\ref{Shah} we have: 
$$
\frac{\lambda_0}{-\lambda_{n-1}}=\frac{\lambda_0}{|\lambda_{n-1}|}=q.
$$
By combining this with Lemma~\ref{Hoffman} we obtain the lower bound $q+1\leq \chi(\sing{2}{\FF_q})$.
In order to establish the upper bound, we will exhibit a 
proper coloring of $\chi(\sing{2}{\FF_q})$ with $8(q+1)$ 
colors. 
Let $  \Sigma=\{0,1,-1\}$ and $ \lambda: \FF_q  \to \Sigma$ be a function satisfying $ \lambda(0)=0$ and $ \lambda(x) \neq  \lambda(-x)$ for all $x \in \FF_q \setminus \{ 0 \}$. Denote by $B$ the subgroup of upper-triangular matrices in $\rG:=\SL_2(\FF_q)$, and let $ \pi: \rG \to \rG/B$ be the canonical map. Define the coloring map $\Theta$ by
$$\Theta: \rG\to \rG/B \times \Sigma \times  \Sigma,\qquad  
X= \begin{pmatrix}
a  & b    \\
 c & d \\
\end{pmatrix} \mapsto (\pi(X), \lambda(c), \lambda(d) ).
$$
We claim that $\Theta$ provides a proper coloring for $\sing{2}{\FF_q}$. Let $X, Y \in \SL_2(\FF_q)$ be such that $\Theta(X)=\Theta(Y)$ and $-1$ is an eigenvalue of $X^{-1}Y$. 
From $\pi(X)=\pi(Y)$ we conclude that $X^{-1}Y \in B$. It is easy to see that every element in $B$ with an eigenvalue equal to $-1$ is of the form
$ \begin{pmatrix}
-1  & t    \\
 0 & -1 \\
\end{pmatrix} $
for some $t \in \FF_q$. Write $X= \begin{pmatrix}
a  & b    \\
 c & d \\
\end{pmatrix}$. Then
$$ Y= \begin{pmatrix}
a  & b    \\
 c & d \\
\end{pmatrix} \begin{pmatrix}
-1  & t    \\
 0 & -1 \\
\end{pmatrix}  = \begin{pmatrix}
-a  & at-b    \\
 -c & ct -d \\
\end{pmatrix}.$$

Since $\Theta(X)=\Theta(Y)$, we have $ \lambda(c)= \lambda(-c)$, which implies that 
$c=0$. Now, we must also have $ \lambda(d)= \lambda(ct-d)= \lambda(-d)$, which 
implies that $d=0$, which is a contradiction. So $\Theta$ provides a
proper coloring and then (noting that $ \lambda(c)=\lambda(d)=0$ cannot occur), we obtain
$$\chi(\sing{2}{\FF_q}) \le 8|\rG/B|= 8(q+1).$$
\end{proof} 
When $-1$ is a quadratic non-residue in $\FF_q$ one can improve the upper bound to $2(q+1)$. Indeed let $H=(\FF_q^{\ast})^2$ be the quadratic residue subgroup which does not contain $-1$. Now consider the following subgroup of $\rG:=\SL_2(\FF_q)$:
\[ B'= \left\{ \begin{pmatrix}
x  & y    \\
 0 & x^{-1} \\
\end{pmatrix}: x \in H, y \in \FF_q \right\}. \]
Let $ \pi: \rG \to \rG/B'$ to be canonical quotient map. We claim that $ \pi$ provides a proper coloring. If $ \pi(X)= \pi(Y)$,
then $X^{-1}Y \in B'$, which implies that the eigenvalues of $X^{-1}Y$ are distinct from $-1$. Hence 
$\chi( \Gamma_2( \FF_q)) \le |\rG/B'|= 2(q+1)$.
In a similar fashion, one obtains the upper bound $4(q+1)$ if 
$-1$ is not a fourth power in $\FF_q^{\ast}$. 
\section{Non-singular graphs over rings}\label{sec-hyperbola}
 This section is devoted to the proof of Theorem \ref{Klo-SL}. 
 Let $R$ be a commutative ring with $1$. The hyperbola graph over $R$, denoted by $\hgraph{R}$, is  defined by 
\[ \hgraph{R}=\Cay(R^2, \mathcal{S}), \qquad \mathcal{S}=\{(x,y)\in R\times R: xy=1\}.\]
\begin{proposition}\label{prop:embed}
Let $R$ be a commutative ring in which $2$ is invertible. Then there exists a subset 
$A \subseteq \SL_2(R)$ of vertices of $\sing{2}{R}$ such that the induced subgraph on $A$ is isomorphic to $\hgraph{R}$.  
\end{proposition}
\begin{proof}
For $x,y\in R$ define
\begin{equation}\label{axy}
a_{x,y}:= \begin{pmatrix}
1  & -2x    \\
 0 & 1 \\
\end{pmatrix} \begin{pmatrix}
1  & 0    \\
 2y & 1 \\
\end{pmatrix} = \begin{pmatrix}
1-4xy  & -2x    \\
 2y & 1 \\
\end{pmatrix} \in \SL_2(R).
\end{equation}
A simple computation shows that $\det(a_{x_1,y_1}+a_{x_2,y_2})=4-4(x_2-x_1)(y_2-y_1).$
Hence, vertices $a_{x_1,y_1}$ and $a_{x_2,y_2}$ of $\sing{2}{R}$ are adjacent if and only if $(x_2-x_1)(y_2-y_1)=1$. Note also that since $2$ is invertible, the map $(x,y) \mapsto a_{x,y}$ is a injective. This implies that the set 
$A= \{ a_{x,y}: x,y \in R \}$
 fulfills the requirements.  
\end{proof}
\begin{corollary}\label{comp-KL} For a commutative ring $R$ in which $2$ is invertible, we have $\chi(\sing{2}{R}) \ge \chi(\hgraph{R})$. 
\end{corollary} 
 The upper bound is rather straightforward; the proof of the lower bound, however, relies on Estermann-Weil bounds for the Kloosterman sums. 
\begin{lemma}\label{dec}
Let $n \ge 2$, and let $R$ be a ring with a proper ideal $\mathfrak{a}$ such that $2$ is not a zero divisor in $R/\mathfrak{a}$. Then
$\chi(\sing{n}{R})\leq \chi( \sing{n}{R/\mathfrak{a}})$ and $\chi(\hgraph{R}) \leq \chi( \hgraph{R/\mathfrak{a}}).$
\end{lemma}
\begin{proof}
Let $\pi:\SL_n(R)\to \SL_n(R/\mathfrak{a})$ be the group homomorphism induced by the natural ring homomorphism $R\to R/\mathfrak{a}$. Let $V(\sing{n}{R/\mathfrak{a}})$ denote the vertex set of the graph and consider a proper coloring $\Theta: V(\sing{n}{R/\mathfrak{a}}) \to [k]$, and define $\widetilde{ \Theta  }= \Theta \circ \pi: V(\sing{n}{R})
\to [k]$. We claim that $ \widetilde{\Theta}$ is a proper coloring of the graph $\sing{n}{R}$.
To see this, assume that $A$ and $B$ form an edge in 
$\sing{n}{R}$. Then $\det(A+B)=0$, implying that 
$\det( \pi(A)+\pi(B))=0$. If $\pi(A) \neq \pi(B)$, then we are done. If not, we have $2^n\det(\pi(A))=0$, which implies that $\det(\pi(A))=0$, contradicting the assumption that $\mathfrak{a}$ is proper. The first inequality follows from here. 
The second inequality can be proven 
in a similar way. 
\end{proof}
We will also need the following straightforward facts about $\hgraph{\ZZ/p^n\ZZ}$. 

\begin{lemma}\label{lemma=Gpn} Let $p\geq 5$ be a prime number and $n\geq 1$ a positive integer. Then the Cayley graph $\hgraph{\ZZ/p^n\ZZ}$ is a non-bipartite, connected $(p^n-p^{n-1})$-regular graph. 
 
\end{lemma}
\begin{proof} Obviously $\hgraph{\ZZ/p^n\ZZ}=\Cay(\ZZ/p^n\ZZ,\mathcal{S})$ is $(p^n-p^{n-1})$-regular graph since 
$$\mathcal{S}=\{(x,y)\in (\ZZ/p^n\ZZ)^2: xy=1\},$$
 has $(p^n-p^{n-1})$ elements. To show that $\hgraph{\ZZ/p^n\ZZ}$ is connected, we prove that $\mathcal{S}$ generates the additive group $(\ZZ/p^n\ZZ)^2$. Clearly $v_1=(1,1), v_2=(2,1/2) \in \mathcal{S}$ and
$
\det\begin{pmatrix}
1 & 1\\
2 & 1/2
\end{pmatrix}=-3/2,$
which is a unit if $p\geq 5$. Hence $\{v_1,v_2\}$ generates $(\ZZ/p^n\ZZ)^2$. Finally notice that for $0\leq i\leq p^n-1$, the vertices $(i,i)$ and $(i+1,i+1)$ are adjacent. Since $p$ is an odd prime $\hgraph{\ZZ/p^n\ZZ}$ contains an odd cycle and so 
$\hgraph{\ZZ/p^n\ZZ}$ is not a bipartite graph.
\end{proof}

Let $u,v$ be two integers and $m$ a positive integer. The associated Kloosterman sum is defined by
$$
\Kl(u,v,m):=\sum_{\substack{x=1\\ \gcd(x,m)=1}}^{m} \exp{\left(\frac{2\pi i (ux+vx^*)}{m}\right)},
$$
where $x^*$ is the inverse of $x$ modulo $m$. By the Estermann-Weil bound~\cite{Estermann}, for $p\geq 3$ we have
\begin{equation}\label{Kloosterman}
|\Kl(u,v,p^n)|\leq 2\gcd(u,v,p^n)^{1/2}p^{n/2},
\end{equation}
\begin{proof}[Proof of Theorem \ref{Klo-SL}]
Let $\lambda_0\geq \lambda_1\geq\dots\geq \lambda_{p^{2n}-1}$ be the spectrum of $\hgraph{\ZZ/p^n\ZZ}$. From Lemma~\ref{lemma=Gpn} we have $ \lambda_0=p^n-p^{n-1}$ and  
\begin{equation}\label{strict}
\max\left\{|\lambda_1|,|\lambda_{p^{2n}-1}|\right\}< p^n-p^{n-1}.
\end{equation}
 Since the $(\ZZ/p^n\ZZ)^2$ is an abelian group then all of its irreducible representations are one-dimensional. So, by Theorem~\ref{Shah}, for each $0\leq i\leq p^{2n}-1$ there exists two integers $1\leq u_i,v_i\leq p^n$ such that
$$
\lambda_i=\sum_{\substack{(x,y)\in (\ZZ/p^n\ZZ)^2\\ xy=1}}\exp{\left(\frac{2\pi i(u_ix+v_iy)}{p^n}\right)}=\Kl(u_i,v_i,p^n).
$$ 
By~\eqref{strict} for $1\leq i\leq p^{2n}-1$ we have $\gcd(u_i,v_i,p^n)\leq p^{n-1}$.  
From~\eqref{Kloosterman} we have 
$$
|\lambda_i|\leq 2p^{(n-1)/2}p^{n/2}=2p^{n-1/2}, \quad 1\leq i\leq p^{2n}-1.
$$  
This implies that 
$
\max\left\{|\lambda_1|,|\lambda_{p^{2n}-1}|\right\}\leq 2p^{n-1/2}
$.
Therefore by Lemma~\ref{Sar} we can deduce that
$$
\chi(\hgraph{\ZZ/p^n\ZZ})\geq \frac{p^{n}-p^{n-1}}{\max\{|\lambda_1|,|\lambda_{p^{2n}-1}|\}}\geq \frac{p^{n}-p^{n-1}}{2p^{n-1/2}}\geq \frac{\sqrt{p}}{4}.
$$
Now, from Corollary \ref{comp-KL}, we have
$\chi(\sing{2}{\ZZ/p^n\ZZ})\geq \chi(\hgraph{\ZZ/p^n\ZZ})\geq \sqrt{p}/4.$ The upper bound 
immediately follows from Lemma~\ref{dec} for $\mathfrak{a}=p^{n-1}\ZZ/p^n\ZZ$.
\end{proof}
\section*{Acknowledgement} During the completion of this work, M.B. was supported by a postdoctoral fellowship from the University of Ottawa. He wishes to thank
his supervisors Vadim Kaimanovich, Hadi Salmasian and Kirill Zainoulline. The authors are grateful to Amin Bahmanian, Boris Bukh and Mike Newman with whom the authors discussed various parts of this paper. M.B would like to thank Alborz Fazaeli for providing him a program to compute chromatic number of small graphs.
K.M-K would like to specially thank Saieed Akbari for introducing him to this problem, and also mentioning the reference~\cite{Tomon}. Authors would like to especially thanks the referees for several detailed comments that lead 
to improving the exposition of the paper and correcting some inaccuracies. 
\bibliographystyle{plain}

\begin{thebibliography}{10}

\bibitem{AJF}
S.~Akbari, M.~Jamaali, and S.~A. Seyed~Fakhari.
\newblock The clique numbers of regular graphs of matrix algebras are finite.
\newblock {\em Linear Algebra Appl.}, 431(10):1715--1718, 2009.

\bibitem{Alon13}
Noga Alon.
\newblock The chromatic number of random {C}ayley graphs.
\newblock {\em European J. Combin.}, 34(8):1232--1243, 2013.

\bibitem{Anderson}
David~F. Anderson and Ayman Badawi.
\newblock The total graph of a commutative ring.
\newblock {\em J. Algebra}, 320(7):2706--2719, 2008.

\bibitem{Babai3}
L\'aszl\'o Babai.
\newblock Chromatic number and subgraphs of {C}ayley graphs.
\newblock pages 10--22. Lecture Notes in Math., Vol. 642, 1978.

\bibitem{Babai}
L{\'a}szl{\'o} Babai.
\newblock Spectra of {C}ayley graphs.
\newblock {\em J. Combin. Theory Ser. B}, 27(2):180--189, 1979.

\bibitem{Pyber}
L{\'a}szl{\'o} Babai, Nikolay Nikolov, and L{\'a}szl{\'o} Pyber.
\newblock Product growth and mixing in finite groups.
\newblock In {\em Proceedings of the {N}ineteenth {A}nnual {ACM}-{SIAM}
  {S}ymposium on {D}iscrete {A}lgorithms}, pages 248--257. ACM, New York, 2008.

\bibitem{BKMS}
Mohammad Bardestani, Camelia Karimianpour, Keivan Mallahi-Karai, and Hadi
  Salmasian.
\newblock Faithful representations of {C}hevalley groups over quotient rings of
  non-{A}rchimedean local fields.
\newblock {\em Groups Geom. Dyn.}, 11(1):57--74, 2017.

\bibitem{BMK}
Mohammad Bardestani and Keivan Mallahi-Karai.
\newblock On a generalization of the {H}adwiger-{N}elson problem.
\newblock {\em Israel J. Math.}, 217(1):313--335, 2017.

\bibitem{Bollobas}
B{\'e}la Bollob{\'a}s.
\newblock {\em Modern graph theory}, volume 184 of {\em Graduate Texts in
  Mathematics}.
\newblock Springer-Verlag, New York, 1998.

\bibitem{Breuillard}
Emmanuel Breuillard.
\newblock A brief introduction to approximate groups.
\newblock In {\em Thin groups and superstrong approximation}, volume~61 of {\em
  Math. Sci. Res. Inst. Publ.}, pages 23--50. Cambridge Univ. Press, Cambridge,
  2014.

\bibitem{Cam}
Peter~J. Cameron.
\newblock Research problems from the {BCC}22.
\newblock {\em Discrete Math.}, 311(13):1074--1083, 2011.

\bibitem{Sarnak}
Giuliana Davidoff, Peter Sarnak, and Alain Valette.
\newblock {\em Elementary number theory, group theory, and {R}amanujan graphs},
  volume~55 of {\em London Mathematical Society Student Texts}.
\newblock Cambridge University Press, Cambridge, 2003.

\bibitem{Shahshahani}
Persi Diaconis and Mehrdad Shahshahani.
\newblock Generating a random permutation with random transpositions.
\newblock {\em Z. Wahrsch. Verw. Gebiete}, 57(2):159--179, 1981.

\bibitem{Dorn}
Larry Dornhoff.
\newblock {\em Group representation theory. {P}art {A}: {O}rdinary
  representation theory}.
\newblock Marcel Dekker, Inc., New York, 1971.
\newblock Pure and Applied Mathematics, 7.

\bibitem{Estermann}
T.~Estermann.
\newblock On {K}loosterman's sum.
\newblock {\em Mathematika}, 8:83--86, 1961.

\bibitem{Jarden}
Michael~D. Fried and Moshe Jarden.
\newblock {\em Field arithmetic}, volume~11 of {\em Ergebnisse der Mathematik
  und ihrer Grenzgebiete. 3. Folge. A Series of Modern Surveys in Mathematics
  [Results in Mathematics and Related Areas. 3rd Series. A Series of Modern
  Surveys in Mathematics]}.
\newblock Springer-Verlag, Berlin, third edition, 2008.
\newblock Revised by Jarden.

\bibitem{Gowers}
W.~T. Gowers.
\newblock Quasirandom groups.
\newblock {\em Combin. Probab. Comput.}, 17(3):363--387, 2008.

\bibitem{LS}
Vicente Landazuri and Gary~M. Seitz.
\newblock On the minimal degrees of projective representations of the finite
  {C}hevalley groups.
\newblock {\em J. Algebra}, 32:418--443, 1974.

\bibitem{LangWeil}
Serge Lang and Andr{\'e} Weil.
\newblock Number of points of varieties in finite fields.
\newblock {\em Amer. J. Math.}, 76:819--827, 1954.

\bibitem{LPS}
A.~Lubotzky, R.~Phillips, and P.~Sarnak.
\newblock Ramanujan graphs.
\newblock {\em Combinatorica}, 8(3):261--277, 1988.

\bibitem{Lubotzky}
Alexander Lubotzky.
\newblock {\em Discrete groups, expanding graphs and invariant measures},
  volume 125 of {\em Progress in Mathematics}.
\newblock Birkh\"auser Verlag, Basel, 1994.
\newblock With an appendix by Jonathan D. Rogawski.

\bibitem{Malle}
Gunter Malle and Donna Testerman.
\newblock {\em Linear algebraic groups and finite groups of {L}ie type}, volume
  133 of {\em Cambridge Studies in Advanced Mathematics}.
\newblock Cambridge University Press, Cambridge, 2011.

\bibitem{Margulis}
G.~A. Margulis.
\newblock Explicit group-theoretic constructions of combinatorial schemes and
  their applications in the construction of expanders and concentrators.
\newblock {\em Problemy Peredachi Informatsii}, 24(1):51--60, 1988.

\bibitem{Steinberg}
Robert Steinberg.
\newblock {\em Lectures on {C}hevalley groups}.
\newblock Yale University, New Haven, Conn., 1968.
\newblock Notes prepared by John Faulkner and Robert Wilson.

\bibitem{Tao}
Terence Tao.
\newblock Mixing for progressions in nonabelian groups.
\newblock {\em Forum Math. Sigma}, 1:e2, 40, 2013.

\bibitem{Tomon}
Istv{\'a}n Tomon.
\newblock On the chromatic number of regular graphs of matrix algebras.
\newblock {\em Linear Algebra Appl.}, 475:154--162, 2015.

\end{thebibliography}

\end{document}